\documentclass[12 pt]{amsart}

\usepackage{amsthm}
\usepackage{amssymb}
\usepackage{latexsym}
\usepackage{url}
\usepackage{amsmath}
\usepackage{verbatim}
\usepackage{graphicx}
\usepackage{epsf}
\usepackage{enumerate}
\usepackage{mathabx}

\usepackage{geometry}
\usepackage{hyperref}
\hypersetup{breaklinks=true,
pagecolor=white,
colorlinks=true,
citecolor=black,%
filecolor=black,%
linkcolor=black,%
urlcolor=black
}

\title[Zone and double zone diagrams]{Zone and double zone diagrams in abstract spaces}
\author{Daniel Reem and Simeon Reich}

\address{Department of Mathematics, The Technion - Israel Institute of 
Technology, 32000 Haifa, Israel. Current  address (July 2011) of Daniel Reem: Department of Mathematics, University of Haifa, Mount Carmel, 31905 Haifa, Israel. }
\email{dream@tx.technion.ac.il; sreich@tx.technion.ac.il}
\subjclass[2000]{06B23, 47H10, 51K99, 54E35} \keywords{ Dominance region,
double zone diagram, fixed point,
$m$-space, metric space, partially ordered set, trisector, zone
diagram}

\newtheorem{thm}{Theorem}[section]
\newtheorem{cor}[thm]{Corollary}
\newtheorem{lem}[thm]{Lemma}

\newtheorem{defin}[thm]{Definition}

\theoremstyle{definition}
\newtheorem{expl}[thm]{Example}

\newtheorem{rem}[thm]{Remark}

\newcommand{\dom}{\textnormal{dom}}
\newcommand{\Dom}{\textnormal{Dom}}

\newcommand{\R}{\mathbb{R}}

\newcommand{\N}{\mathbb{N}}

\begin{document}
\begin{abstract}
A zone diagram is a relatively new concept which was first 
defined and studied  by T. Asano, J.
Matou{\v{s}}ek and T. Tokuyama.
It can be interpreted  as a state of  equilibrium between several mutually hostile kingdoms.
Formally, it is a fixed point of a certain mapping. These authors
considered the Euclidean plane and proved the existence and uniqueness of zone
diagrams there. In the present paper we generalize this concept in various
ways. We consider general sites in $m$-spaces (a simple generalization of
metric spaces) and prove several existence and (non)uniqueness results in 
this setting. In contrast to previous works, our (rather simple) proofs
are based on purely order theoretic arguments.  Many explicit examples are given, and some of them illustrate new phenomena which occur in the general case. We also re-interpret zone diagrams as a stable configuration in a certain combinatorial game, and provide an
algorithm for finding this configuration in a particular case.
\end{abstract}

\maketitle

\section{Introduction and notation}
A zone diagram is a relatively new concept 
 which can be interpreted  as a state of  equilibrium
 between several mutually hostile kingdoms. More formally, let $(X,d)$ be a metric space, and suppose $P=(P_k)_{k\in K}$ is a given tuple of nonempty sets in $X$. A zone diagram with respect to $P$ is a tuple $R=(R_k)_{k\in K}$ of nonempty sets such that each $R_k$ is the set of all $x\in X$ which are closer to $P_k$
than to $\bigcup_{j\neq k} R_j$. In other words, $R$ is a fixed point of a certain mapping (called
the Dom mapping). Neither its existence, nor its uniqueness, are obvious 
{\em a priori}.

The concept of a zone diagram was first defined and studied
by T. Asano, J. Matou{\v{s}}ek and T. Tokuyama \cite{AMT2,AMTn}, in the case
where $X$ was the Euclidean plane, $K$ was finite, each $P_k$ was a single
point and all these points were different. They proved
the existence and uniqueness of a zone diagram in this case. Their proofs rely heavily  on
the above setting.

In this paper we generalize this concept in various ways. As we have already mentioned, we
consider general tuples of sets $P=(P_k)_{k\in K}$, and
general metric spaces. In fact, we consider  a more general setting
($m$-spaces; see Section
\ref{sec:exmetric}).  One of the advantages of this generalization,
besides its leading to general results, is that it yields a better
understanding of the concept, and enables us to give plenty of explicit examples
of zone diagrams, a task which is quite hard in the case of
singleton-site zone diagrams in the Euclidean plane. These examples 
illustrate some new phenomena  which occur in the general case. Moreover, this generalization also
opens up new possibilities for applying this concept in other parts of
mathematics and elsewhere.

Exact definitions, as well as several examples, are given in Sections \ref{sec:examples} and
\ref{sec:exmetric}. In Section \ref{sec:combi} we re-interpret the concept of a zone
diagram as a stable configuration in a certain combinatorial game.

Our main existence results are Theorem \ref{thm:n_is_2}, which shows
the existence of a zone diagram of 2 sites in any $m$-space, and
Theorem \ref{thm:DoubleZoneDiagram} which shows the existence of a
double zone diagram (a fixed point of the second iteration $\Dom^2$)
of any number (possibly infinite) of sites. Our method, which is different from the methods
described in \cite{AMT2} and \cite{AMTn}, is based on the
Knaster-Tarski fixed point theorem for monotone (increasing)
mappings. [One, in fact two, of the arguments in \cite{AMTn} do make
use of a fixed point theorem (the Schauder fixed point theorem), but
for continuous mappings rather than monotone ones.] It can be seen
that our (rather simple) proofs have a purely order theoretic character; 
there is
no need to take into account any other considerations  (algebraic,
topological, analytical, etc.). As a corollary we obtain the
existence of a trisector in any Hilbert space, and the proof can be
considered ``conceptual''; see  Remark ~\ref{rem:trisector} 
in Section \ref{sec:existence}.

In Section \ref{sec:uniqueness} we discuss the uniqueness question. In general,
 there can be several zone diagrams, but we do present several necessary and
sufficient conditions for uniqueness. In Section \ref{sec:algorithm} we describe a simple algorithm for constructing a zone diagram of
order 2 and a double zone diagram of any order in the case where $X$ is a finite set.
We conclude the paper by formulating some interesting open problems.

We end this introduction with a couple of words about notation.
Throughout the text we will make use of tuples, the components of which are sets.
Every operation or relation between such tuples, or on a single tuple, is done component-wise. Hence, for example, if $K\neq \emptyset$ is a set of indices, and if $R=(R_k)_{k\in K}$ and $S=(S_k)_{k\in K}$ are two tuples
 of sets, then $R\bigcap S=(R_k\bigcap S_k)_{k\in K}$,  $\overline{R}=(\overline{R_k})_{k\in K}$,
and $R\subseteq S$ means $R_k\subseteq S_k$ for each $k\in K$.
The tuple $(X)_{k\in K}$ is the tuple all the components of which are equal
to $X$. Given a set $X$, we denote by   $\mathcal{P}^{*}(X)$ the set of all
nonempty subsets of $X$, and by $|X|$ the cardinal number of $X$.

\section{Definitions and examples}\label{sec:examples}
Zone diagrams can be naturally defined in any metric
space. In Section \ref{sec:existence} we will prove a theorem which
ensures the existence of zone diagram of two sites in any metric
space. Surprisingly, the proof can be carried
over to a more general setting, which we call $m$-spaces.
However, since the latter concept seems to be new, and since the
concept of a zone diagram is best understood in the context of metric
spaces, we  discuss it first in this context. The
corresponding generalization will be carried out in Section
\ref{sec:exmetric}.
\begin{defin}
Let $(X,d)$ be a metric space. 
 For any $P,A\in \mathcal{P}^{*}(X)$, the dominance region
$\dom(P,A)$ of $P$ with respect to $A$ is the set of all $x\in X$
which are closer to $P$ than to $A$,
 i.e., it is the set
\begin{equation*}
\dom(P,A)=\{x\in X: d(x,P)\leq d(x,A)\}.
\end{equation*}
The function $\dom:\mathcal{P}^{*}(X) \times \mathcal{P}^{*}(X)\to 
\mathcal{P}^{*}(X)$ is called the dom mapping.
\end{defin}
For example, if $a$ and $p$ are two different points in a Hilbert space $X$, and if $P=\{p\}$ and $A=\{a\}$, then $\dom(P,A)$ is the half-space containing $P$ determined by the hyperplane passing through the middle of the line segment $[p,a]$ and perpendicular to it; $\dom(A,P)$ is the other half-space.
\begin{defin}
Let $(X,d)$ be a metric space and let $K$ be a set of at least 2 elements (indices), possibly
infinite. Given a
 tuple $(P_k)_{k\in K}$ of nonempty subsets
$P_k\subseteq X$, a zone diagram  with respect to
that tuple is a tuple $R=(R_k)_{k\in K}$ of nonempty subsets
$R_k\subseteq X$ such that
\begin{equation*}
R_k=\dom(P_k,\textstyle{\underset{j\neq k}\bigcup R_j})\quad \forall k\in K.
\end{equation*}
In other words, if we define $X_k=\mathcal{P}^{*}(X)$, then a zone diagram
is a fixed point of the mapping $\Dom:\underset{{k\in K}}\bigtimes
X_k\to \underset{{k\in K}}\bigtimes
X_k$, defined by 
\begin{equation}\label{eq:TZoneDef}
\Dom(R)=(\dom(P_k,\textstyle{\underset{j\neq k}\bigcup R_j}))_{k\in K}.
\end{equation}
If the second iteration $\Dom^2=\Dom\circ \Dom$ has a
fixed point $R$, we say that $R$ is a double zone diagram in $X$.
\end{defin}
If we interpret each $R_k$ as an ancient kingdom, and each $P_k$ as a 
site or a collection of sites in $R_k$ (cities, army camps, islands, 
etc.), then a 
zone diagram is a configuration in which each kingdom $R_k$ consists of all the points $x\in X$ which are closer to $P_k$ than to the other kingdoms.   This can be regarded as a state of equilibrium between the kingdoms in the following sense.
 Suppose the kingdoms are mutually hostile. In particular, each kingdom has to defend its borders against attacks from the other kingdoms. Due to various considerations (resources, field conditions, etc.), the defending army is usually situated only in (part of) the sites $P_k$ (unless the kingdom moves forces to attack another kingdom), and each $P_k$ remains unchanged.  Hence, if $(R_k)_{k\in K}$ is a zone diagram, then each point in each kingdom can be defended at least as fast as it takes to attack it from any other kingdom, and no kingdom can enlarge its territory without violating this condition. (But see Example \ref{ex:simple} and Example \ref{ex:ProbInfinite} for some non-realistic counterexamples.)

 A double zone diagram $R$ is a different state of equilibrium between the kingdoms: now each
 kingdom $R_k$ consists of all points $x\in X$ which are closer to $P_k$ than to the union $\bigcup_{j\neq k}(\Dom R)_j$. We note that any zone diagram $R$ is obviously a double zone diagram
since $\Dom^2 R=\Dom(\Dom R)=\Dom R=R$, but
the converse is not necessarily true as the following example shows.

\begin{expl}\label{ex:simple}
Let $X=\{-1,0,1\}$ be a subset of $\mathbb{R}$ with the usual
 metric, and let $P_1=\{-1\},\,P_2=\{1\}$. If we let
$R_1=P_1$ and $R_2=\{0,1\}$, then $-1$ represents all the points in
$X$ closer to $P_1$ than to $R_2$, and $0,1$ are all the points in
$X$ closer to $P_2$ than to $R_1$, so
\begin{equation*}
\dom(P_1,R_2)=\{-1\}=R_1\quad \textrm{and}\quad \dom(P_2,R_1)=\{0,1\}=R_2,
\end{equation*}
i.e., $R=(R_1,R_2)$ is a zone diagram in $X$. Similarly, $S=(\{-1,0\},\{1\})$ is a different zone
diagram in $X$, and this shows that uniqueness does not hold in
general. However, if we replace the point $0$ with a point $a \in (0,1)$, 
then the modified $R$ is still a zone diagram, and it is unique. Indeed,
suppose $Z=(Z_1,Z_2)$ is another zone diagram. Obviously,
$P_1\subseteq \dom(P_1,Z_2)$ and $P_2\subseteq \dom(P_2,Z_1)$, so
any $x\neq -1$ is closer to $P_2$, and hence to $Z_2$, than to
$Z_1$. Thus $Z_1$ must be $P_1$, but then
$Z_2=\dom(P_2,Z_1)=\{a,1\}$, so $Z=R$.

Already in the original example, where $X = \{-1,0, 1\}$, it can be seen 
that a double zone diagram is not
necessarily a zone diagram, since
\begin{multline*}
\Dom^2(P_1,P_2)=(\dom(P_1,\dom(P_2,P_1)),\dom(P_2,\dom(P_1,P_2))\\
=(\dom(\{-1\},\{0,1\}),\dom(\{1\},\{-1,0\}))=(\{-1\},\{1\})=(P_1,P_2),
\end{multline*}
but $(P_1,P_2)$ is not a zone diagram because $P_1\neq
\{-1,0\}=\dom(P_1,P_2)$.  However, if we replace the point $0$ with 
a point $a \in (0,1)$,
then the double zone diagram and the zone diagram coincide.
\end{expl}
\begin{figure}[t]
\begin{minipage}[t]{1\textwidth}
\begin{center}
\scalebox{0.6} {\includegraphics{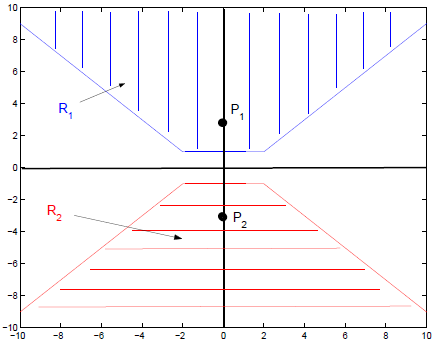}}
\end{center}
 \caption{ Example \ref{ex:maxnorm}}
\label{fig:maxnorm}
\end{minipage}%
\end{figure}

\begin{expl}\label{ex:maxnorm}
Let $X=\mathbb{R}^2$ with the max norm
$|x|=|(x_1,x_2)|=\max\{|x_1|,|x_2|\}$, and let
$P_1=\{(0,3)\},\,P_2=\{(0,-3)\}$. Suppose $f:\mathbb{R}\to
\mathbb{R}$ is the function defined by
\begin{equation*}
f(x_1)=\left\{\begin{array}{ll}
                -x_1-1 & x_1\leq -2 \\
                1 & x_1\in [-2,2] \\
                x_1-1 & x_1\geq 2.
              \end{array}
\right.
\end{equation*}
If $R_1$ and $R_2$ are the domains above and below the graphs of
$f$ and $-f$ respectively, i.e., $R_1=\{(x_1,x_2):x_2\geq
f(x_1)\}$ and $R_2=\{(x_1,x_2):x_2\leq -f(x_1)\}$, then
$R=(R_1,R_2)$ is a zone diagram in $X$; see Figure
\ref{fig:maxnorm}.

Indeed, if $x=(x_1,x_2)$ is in $R_1$, then there are 3
possibilities: $x_1\leq -2,\,x_1\in [-2,2]$ and $x_1\geq 2$. The
third case is treated in the same way as the first one, so it suffices to consider
only the first and the second. 
In the first case an elementary calculation shows that
$d(x,R_2)=d(x,(-2,-1))$. Hence $d(x,R_2)\geq d(x,P_1)$, because
$d(x,(-2,-1))\geq x_2+1\geq\max\{-x_1,|x_2-3|\}=d(x,P_1)$. In the
second case either $x_2\in [1,5]$ and then $d(x,P_1)\leq2\leq x_2+1=
d(x,R_2)$, or $x_2>5$ and then $d(x,P_1)=x_2-3<x_2+1= d(x,R_2)$.
Thus, in every case $d(x,P_1)\leq d(x,R_2)$, i.e., $R_1\subseteq
\dom(P_1,R_2)$.

On the other hand, suppose $d(x,P_1)\leq d(x,R_2)$ and assume to the
contrary that $x\notin R_1$. Then $x_2\geq 0$, for otherwise
$d(x,R_2)\leq d(x,R_1)<d(x,P_1)$. In addition, $|x_1|\leq 2$,
because otherwise, if, for example, $x_1<-2$, then using the fact that
$x\notin R_1$ implies $x_2<-x_1-1$, we arrive at the inequality
$d(x,R_2)\leq d(x,(-2,-1))=\max\{-x_1-2,x_2+1\}<|x_1|\leq
d(x,P_1)$. So $|x_1|\leq 2$, but then $d(x,R_2)\leq 2$, and since $x\notin R_1$,
we have, in fact,  $d(x,R_2)<2<d(x,P_1)$, a contradiction. Therefore
$\dom(P_1,R_2)\subseteq R_1$ and we get equality. In the same
way, $R_2=\dom(P_2,R_1)$.

A reader who is familiar with the concept of a trisector (see \cite{AMT2}), may have already noticed that the boundaries of $R_1$ and $R_2$ (denoted by $C_1$ and $C_2$, respectively) represent the 
components of a trisector, i.e., they satisfy the equations $C_1=\{x\in X: d(x,P_1)=d(x,C_2)\}$ and
$C_2=\{x\in X: d(x,P_2)=d(x,C_1)\}$. 
The sets $C_1$ and $C_2$ are indeed the graphs of convex/concave functions, 
but in contrast to the Euclidean case  (\cite[Theorem 2]{AMT2} 
and the discussion following it), these functions have a simple form and 
they are not analytic.
\end{expl}

In our next example the two sites $P_1$ and $P_2$ have a nonempty 
intersection.  

\begin{expl}\label{ex:ProbInfinite}
$X=\mathbb{R}^2$ with the Euclidean norm,
$P_1=\mathbb{Q}\times\{0\}$ and $P_2=(\N\,\bigcup\,(\mathbb{R}\backslash\mathbb{Q}))\times\{0\}$.
At first sight it seems that either a zone diagram may not exist at all, or that
it may be pathological if it does exist. Nevertheless, a simple
check shows that $(\R\times\{0\},\R^2)$ is a zone diagram, and so is
$(\R^2,\R\times\{0\})$. It is interesting to note that if we consider the infinite
family $P_x=\{(x,0)\},\,x\in \R$, then it is not clear at all whether
now there exists a zone diagram $R=(R_x)_x$. If it does exist, then it is
probably pathological.

The above example can be generalized: if $(X,d)$ is any metric space, and
if the tuple $P=(P_k)_{k\in K}$ has the property that 
$\overline{P_j}=\overline{P_k}\,\,\,\forall k,j\in K$,
then for any $i\in K$ and any tuple $R=(R_k)_{k\in K}$ with the property 
that $R_k=\overline{P_k}\,\,\,\forall k\neq i$ and $R_{i}=X$, we have 
that $R$ is a zone diagram in $X$. Hence, some restrictions on $P$ have 
to be imposed in order
to obtain uniqueness. (For instance, 
that $\inf\{d(\overline{P_k},\overline{P_j}): j\neq k\}>0$ for all $k\in K$; 
it is not clear, however, when this condition is sufficient; see 
Example \ref{ex:simple}.) 
\end{expl}
\begin{figure}[t]
\begin{minipage}[t]{1\textwidth}
\begin{center}
\scalebox{0.6} {\includegraphics{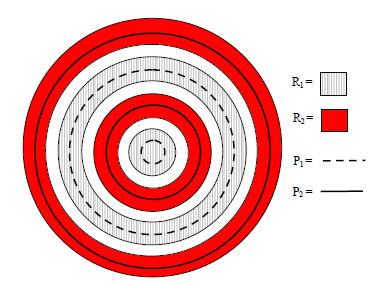}}
\end{center}
\caption{Example \ref{ex:rings}}
 \label{fig:rings}
\end{minipage}%
\end{figure}
\begin{expl}\label{ex:rings}
Let $(X,|\cdot|)$ be any normed space, and let
\begin{equation*}
P_1=\textstyle{\bigcup_{k=0}^{\infty}}\{x\in X: |x|=6k+1\},\quad
P_2=\textstyle{\bigcup_{k=0}^{\infty}}\{x\in X: |x|=6k+4\}.
\end{equation*}
It can easily be checked that $R=(R_1,R_2)$ is a zone diagram in
$X$, where
\begin{equation*}
 R_1=\textstyle{\bigcup_{k=0}^{\infty}}\{x\in X: 6k\leq|x|\leq
6k+2\},\,\,   R_2=\textstyle{\bigcup_{k=0}^{\infty}}\{x\in X: 6k+3\leq|x|\leq 6k+5\}.
\end{equation*}
See Figure \ref{fig:rings}. The zone diagram in this case is
unique; see Section \ref{sec:uniqueness}. Suppose now that we
modify this example by letting $P_k=\{x\in X: |x|=3k+1\},\,k\in\N\,\textstyle{\bigcup}\,\{0\}$. The resulting zone diagram is $R=(R_k)_{k=0}^{\infty}$, where
$R_k=\{x\in X: 3k\leq|x|\leq 3k+2\}$. We obtain the same array of rings
as before, but now each ring represents one and
 only one $R_k$.
\end{expl}

\section{Generalization to $m$-spaces}\label{sec:exmetric}
\begin{defin}
An $m$-space is a pair  $(X,d)$ of a nonempty set $X$ and a
function $d:X^2\to [-\infty,\infty]$ with the property that
\begin{equation}\label{eq:exmetric}
d(x,x)\leq d(x,y)
\,\,\forall x,y \in X. 
\end{equation}
\end{defin}
We call $d$ the distance function, although it is
usually not a true distance function since it is not assumed to be either
symmetric, positive or to satisfy the triangle inequality (take, for 
example, $X=\R$ and
$d(x,y)=\max\{x+1,y\}$).
Given $x\in X$ and $A\in \mathcal{P}^{*}(X)$, we denote
\begin{equation*}
d(x,A)=\inf\{d(x,y):y\in A\},
\end{equation*}
and call it the distance between the point $x$ and the set $A$. All the relevant concepts (the
dom mapping, zone diagrams, etc.) are defined exactly as in the metric
case, but now, however, the interpretations given in that case are less
clear.

We see that the only constraints on $d$ are condition
\eqref{eq:exmetric}, which implies that the natural property
$P\subseteq\dom(P,A)$ will be satisfied, and that its range is
$[-\infty,\infty]$. These requirements alone suffice for ensuring
the existence of zone diagrams of $2$ sites, and actually the
concept of an $m$-space has its origin in an examination of an earlier 
proof we had of the  existence of a zone diagram in the case $n=2$.
\begin{figure}[t]
\begin{minipage}[t]{0.5\textwidth}
\scalebox{0.86} {\includegraphics{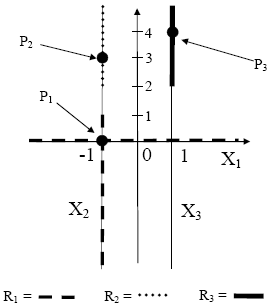}}
 \caption{Example \ref{ex:isolate}}
\label{fig:isolate}
\end{minipage}%
\hfill
\begin{minipage}[t]{0.5\textwidth}
\scalebox{0.64} {\includegraphics{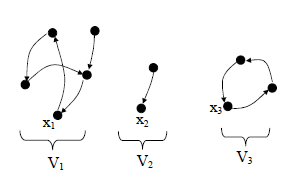}}
\caption{Example \ref{ex:graph}}
 \label{fig:graph}
\end{minipage}%
\end{figure}
We also remark that instead of taking $[-\infty,\infty]$ as the
range of $d$, one can take any totally ordered set which has the
greatest lower bound property, but we will confine ourselves to the
above definition.

We now provide several examples in order to illustrate the concepts of 
$m$-spaces and zone diagrams in this general setting.  
\begin{expl}
Let $X$ be any nonempty set and let $a<b$ be two real numbers. Suppose $d:X^2\to \mathbb{R}$ is defined by $d(x,x)=a<b=d(x,y), \,\, \forall x\neq y$. The function $d$ is usually not a metric. If $(P_k)_{k\in K}$ is any tuple of nonempty and pairwise disjoint
subsets of $X$, then any tuple $(R_k)_{k\in K}$ of nonempty and
pairwise disjoint subsets of $X$ for which $P_k\subseteq R_k$ and
$\bigcup_{k\in K} R_k=X$, is a zone diagram in $X$ of $|K|$ sites.
\end{expl}
\begin{expl}\label{ex:isolate}
Let $X=\bigcup_{i=1}^3 X_i$ where $X_1=\R\times
\{0\},\,X_2=\{-1\}\times \R,\,X_3=\{1\}\times \R$. Define
$d:X^2\to [-\infty,\infty]$ by $d(x,y)=|x-y|$ if $x,y\in X_i$ for
some $i$, and $d(x,y)=\infty$ otherwise. The function $d$ completely
isolates each component  $X_i$ of $X$ in the sense that a point
$x\in X$ ``feels'' (or ``is affected by'') only points from the
component to which it belongs. Let
$P_1=\{(-1,0)\},\,P_2=\{(-1,3)\},\,P_3=\{(1,4)\}$.  Then
$R=(R_k)_{k=1}^3$ is a zone diagram in $X$, where $R_1=(\R\times
\{0\})\bigcup (\{-1\}\times (-\infty,1]),\, R_2=\{-1\}\times
[2,\infty),\,R_3=\{1\}\times [2,\infty)$. See Figure \ref{fig:isolate}.
\end{expl}
\begin{expl}\label{ex:graph}
Let $G=(V,E)$ be a directed graph, and suppose $d:V^2\to
[0,\infty]$ assigns to $(x,y)\in V^2$ the length of the minimal finite
directed path starting from $x$ and ending in $y$ (with $d(x,x)=0$), 
including
$\infty$ if there is no such finite path. Let $G$ be the graph in
Figure \ref{fig:graph}, and
let $P_1=\{x_1,x_2\}$ and $P_2=\{x_2,x_3\}$. Then
$R=(V_1\bigcup \{x_2\},V_2\bigcup V_3)$ and $S=(V_1\bigcup V_2,\{x_2\}\bigcup
V_3)$ are two zone diagrams in $V$.
\end{expl}

\section{A combinatorial interpretation}\label{sec:combi}

In this section (which is not needed for later sections) we describe
a second interpretation of the
concept of a zone diagram as a certain combinatorial game of one player.

The player is given a set of points $X$, a distance function $d:X^2\to [0,\infty)$
and a tuple $(P_k)_{k\in K}$ of nonempty
subsets of $X$. For simplicity we assume that these sets are pairwise
disjoint, that $(X,d)$ is a finite metric space, and that $K$ is finite.
The set $K$ is interpreted as a set of different colors, and each point $x\in X$
can be colored by one of the colors in $K$, or by an additional neutral color. 
In the initial position each set $P_k$ is colored by the color $k$, and
all other points are colored by the neutral color. Let $R_k:=P_k$ and $Q:=X\backslash (\bigcup_{k\in K}P_k)$.

\begin{figure}[t]
\begin{minipage}[t]{1\textwidth}
\begin{center}
\scalebox{0.65} {\includegraphics{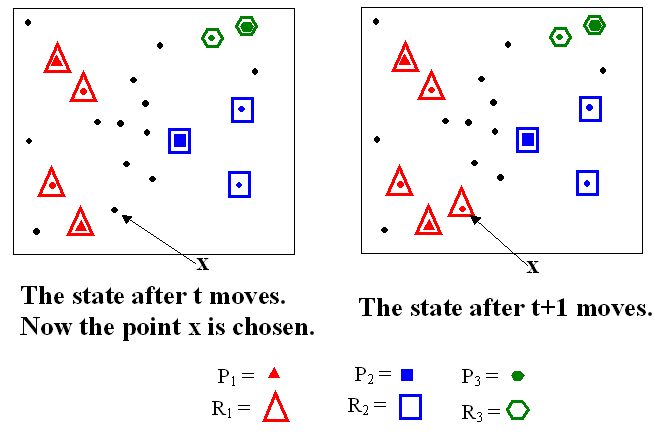}}
\end{center}
 \caption{Illustration of the game}
\label{fig:game}
\end{minipage}%
\end{figure}
Now the game starts. The player chooses a point $x\in Q$ and checks
its position. By this we mean that the player checks whether $x$
belongs to one of the sets $\dom(P_k,\bigcup_{j\neq k}R_j)$. If $x$
does not belong to any $\dom(P_k,\bigcup_{j\neq k}R_j)$, then $x$ is
colored by the neutral color. If it does, then it may happen that
$x$ belongs to several different such sets, corresponding to a
subset $K'\subseteq K$. In this case, if the current color of $x$ is
in $K'$, then the player does not change this color. Otherwise, the
player arbitrarily picks a color $k\in K'$, colors $x$ with it, and
adds $x$ to $R_k$, i.e., the player defines $R_k:=R_k\bigcup\{x\}$.
If the color of $x$ was changed, this means that before the change
either $x$ had a neutral color or it was in some $R_j$, but now it
does not belong to these $R_j$ anymore. In this case the player
removes $x$ from them, i.e., the player defines
$R_j:=R_j\backslash\{x\}$ for all such $j$. The game  continues in
the same manner with the other points in $Q$; see Figure
\ref{fig:game} in which $(X,d)$ is a finite subset of the Euclidean
plane and $|K|=3$.

The goal is to reach a stable configuration $(R_k)_{k\in K}$, i.e., a configuration in which nothing has to be done: for each $x\in Q$, the player does not need to color it by a new color after its position is checked.
A stable configuration is exactly a zone diagram, because
it means that both $\dom(P_k,\bigcup_{j\neq k}R_j)\subseteq R_k$
and $R_k\subseteq \dom(P_k,\bigcup_{j\neq k}R_j)$ for each $k$.
(If $x\in \dom(P_k,\bigcup_{j\neq k}R_j)\backslash R_k$ for a subset of
indices $K'\subseteq K$, then its color is not in $K'$, but it has to be
changed to some $k\in K'$,
and if $x\in R_k\backslash\dom(P_k,\bigcup_{j\neq k}R_j)$,
then either $x\in \dom(P_i,\bigcup_{j\neq i}R_j)$
for some $i\in K'\subseteq K\backslash\{k\}$
and then its color has to be changed to some $i\in K'$,
or $x\notin \dom(P_i,\bigcup_{j\neq i}R_j)$ for
all $i$, and then its color should be changed from $k$ to the neutral
color.)

It is definitely not clear in advance that a stable configuration can be obtained, because of the dynamical character of the game: when changing one $R_k$, by adding/removing a point to/from it, this may affect the other $R_i$, since now $\dom(P_i,\bigcup_{j\neq i}R_j)$ becomes smaller/larger. Hence, even if a point $x$ has already been colored, 
now it is possible that its color is not correct anymore and it will have to be updated. Hence, passing over all of $Q$ once will usually not suffice and it is not clear at all that even infinitely many passes will indeed suffice.

Theorem \ref{thm:n_is_2} ensures the existence of a stable configuration in the case where
 $|K|=2$ and $X,P_1,P_2$ are arbitrary. Using the algorithm described in Section \ref{sec:algorithm},
 such a configuration can be explicitly constructed. We note, however, that the algorithm does not
 show whether the player has a winning strategy, i.e., if and how the player can obtain a stable configuration according to the rules of the game.

\section{Existence}\label{sec:existence}
Our main tool for establishing the existence of a zone diagram of $2$ sites, and of a double zone diagram of any number (possibly infinite) of sites, in any $m$-space (no matter how bizarre the space $(X,d)$ or the sets $P_k$ are) is the Knaster-Tarski fixed point theorem \cite{Knaster,Tarski} which will be stated below.
 We need two definitions before formulating it.
\begin{defin}\label{def:monotone}
Let $Y$ be a nonempty set and suppose $\leq$ is a partial order on $Y$. A
mapping $\,g:Y\to Y$ is called monotone (or isotone, or increasing) if for each $A$ and $B$ in $Y$ the condition $A\leq B$ implies $g(A)\leq g(B)$. The mapping is called antimonotone if $A\leq B$ implies $g(B)\leq g(A)$.
\end{defin}
\begin{defin}\label{def:lattice}
Let $Y$ be a nonempty set and suppose $\leq$ is a partial order on $Y$.
The pair $(Y,\leq)$ is called a complete lattice if any subset $Z$ of $Y$ 
has a least upper bound $\,\bigvee Z \in Y$ and a greatest 
lower bound $\,\,\bigwedge Z \in Y$.
\end{defin}
\begin{thm}\label{thm:CSB}(\textbf{Knaster-Tarski}, \cite[Theorem 1]{Tarski})
Let $(Y,\leq)$ be a complete lattice and let $g:Y\to Y$ be monotone. Then
the set $F$ of all the fixed points of $g$ is nonempty and $(F,\leq)$ is a
complete lattice. In particular, $g$ has two fixed points $m$ and $M$ with the property that $m\leq \mu\leq M$ for any $\mu\in F$. Moreover,
 $m=\bigwedge \{y\in Y: g(y)\leq y\}$ and $M=\bigvee\{y\in Y: y\leq g(y)\}$.
\end{thm}
Both of the sets $\{y\in Y: g(y)\leq y\}$ and $\{y\in Y: y\leq g(y)\}$ are
nonempty
because $\bigvee Y$ belongs to the first and $\bigwedge Y$ to the
second.
It is interesting to note that the proof of 
Theorem \ref{thm:CSB} is elementary and is not
based on the axiom of choice. In order to apply Theorem \ref{thm:CSB} 
we need the following lemma. We note that 
Lemma \ref{lem:DomProperties}\eqref{item:antimonotone} 
generalizes \cite[Lemma 3(ii)]{AMT2}.
\begin{lem}\label{lem:DomProperties} Let $(X,d)$ be an $m$-space. We 
partially order $\mathcal{P}^*(X)$ by inclusion and let $(P_k)_{k \in K}$ 
be a tuple of nonempty subsets in $X$.
\begin{enumerate}[(a)]
\item\label{item:antimonotone} Given $P\in \mathcal{P}^{*}(X)$, the mapping $A \mapsto \textnormal{dom}(P,A)$ is antimonotone.
\item\label{item:Dom2Monotone} The $\Dom$ mapping is antimonotone with respect to
component-wise inclusion. The mapping $\Dom^2$ is monotone.
\item Given $P\in \mathcal{P}^{*}(X)$, the mapping  $A \mapsto \textnormal{dom}(A,P)$ is monotone.
\item\label{item:exmetric} $P\subseteq \dom(P,A)\subseteq X \quad \forall A,P\in \mathcal{P}^{*}(X)$.
\item\label{item:Y} Let
\begin{equation*}
Y_k=\{A\in \mathcal{P}^{*}(X): P_k\subseteq A\},\quad Y=\bigtimes_{k\in K} Y_k.
\end{equation*}
Then $\Dom$ maps $Y$ into $Y$, i.e., $\Dom(W)\in Y$ for any $W\in Y$.
\item Let $Y$ be as above and define on $Y$ the natural partial order $\subseteq$ as follows:
\begin{equation*}
(A_k)_{k\in K}\subseteq (B_k)_{k\in K} \Leftrightarrow A_k\subseteq B_k\quad \forall k\in K.
\end{equation*}
Then $(Y,\subseteq)$ is a complete lattice.
\end{enumerate}
\end{lem}
\begin{proof}
\begin{enumerate}[(a)]
\item If $A\subseteq B$, then $d(x,B)\!=\!\inf\{d(x,y)\!:\! y\in B\}\!\leq\! \inf\{d(x,y)\!:\! y\in A\}\!=\!d(x,A)$, so
   $x\in \dom(P,B)$ implies $d(x,P)\leq d(x,B)\leq d(x,A)$, i.e.,
$x\in \dom(P,A)$.
\item If $R\subseteq S$, then $\bigcup_{j\neq k}R_j\subseteq \bigcup_{j\neq k}S_j$ for all $k\in K$, so
$\dom(P_k,\bigcup_{j\neq k}S_j)\subseteq \dom(P_k,\bigcup_{j\neq i}R_j)$ by
part (\ref{item:antimonotone}), and hence $\Dom(S)\subseteq\Dom(R)$. The second 
assertion follows from the first, since a composition of two antimonotone mappings is a monotone one.
\item  Suppose $A\subseteq B$. Then $d(x,B)\leq d(x,A)$ for all $x\in X$. Hence, if $x\in \dom(A,P)$, then  $d(x,B)\leq d(x,A)\leq d(x,P)$, i.e., $x\in \dom(B,P)$.
\item The second inclusion is obvious and the first one follows immediately from $d(x,P)\leq d(x,x)\leq d(x,y),\,\forall x\in P,y\in A$ by \eqref{eq:exmetric}.
\item The $k$-th component of $\Dom(W)$ contains $P_k$ by part \eqref{item:exmetric}.
\item Let $J\neq\emptyset$ be a subset of $Y$. Then the component-wise intersection $\bigcap_{R\in J}R$ is in $Y$, and it is the greatest lower bound of $J$, and $\bigcup_{R\in J}R$ is the least upper bound of $J$. In addition, $\bigwedge \emptyset=(X)_{k\in K}$ and $\bigvee \emptyset=(P_k)_{k\in K}$.
\end{enumerate}
\end{proof}

\begin{thm}\label{thm:DoubleZoneDiagram}
Let $(X,d)$ be an $m$-space and suppose that $(P_k)_{k\in K}$ is a
tuple of nonempty subsets of $X$. Then there is a double zone diagram
in $X$ with respect to $(P_k)_{k\in K}$. Furthermore, there are double zone diagrams $m$ and $M$ with the property that $m\subseteq \mu\subseteq M$ for any other double zone diagram $\mu$, and they satisfy $m=\Dom(M)$ and $M=\Dom(m)$.
\end{thm}
\begin{proof}
Because of Lemma \ref{lem:DomProperties}, the conditions of Theorem \ref{thm:CSB} are satisfied with $(Y,\subseteq)$ and $g=\Dom^2$. Hence $\Dom^2$ has two (not necessarily different) fixed points $m$ and $M$  which are double zone diagrams by definition, and they  have  the property $m\subseteq \mu\subseteq M$ for any other double zone diagram $\mu$. Finally, since $\Dom(\mu)$ is a double zone diagram for any double zone diagram $\mu$  (because $\Dom^2(\Dom(\mu))=\Dom(\Dom^2(\mu))=\Dom(\mu)$), it follows that $m\subseteq  \Dom(M)$, so by antimonotonicity, $\Dom(m)\supseteq \Dom^2(M)=M$. Therefore $M=\Dom(m)$, and in the same way $m=\Dom(M)$.
\end{proof}

\begin{thm}\label{thm:n_is_2}
Let $(X,d)$ be an $m$-space and let $P_1,P_2\in \mathcal{P}^{*}(X)$. Then there exists a zone diagram 
in $X$ with respect to $(P_1,P_2)$
\end{thm}
\begin{proof}
Let $S=(S_1,S_2)$ be a fixed point of $\Dom^2$, the existence of which is ensured by Theorem \ref{thm:DoubleZoneDiagram}. We have
\begin{equation*}
(S_1,S_2)=\Dom^2(S_1,S_2)=(\dom(P_1,\dom(P_2,S_1)), \dom(P_2,\dom(P_1,S_2))).
\end{equation*}
Let $R_1:=S_1,\,R_2:=\dom(P_2,R_1)$. Then 
$R_1=S_1=\dom(P_1,\dom(P_2,S_1))=\dom(P_1,R_2)$, and hence $R=(R_1,R_2)$ 
is a zone diagram in $X$.
\end{proof}
\begin{rem}\label{rem:trisector} By a simple argument (repeating 
word by word the proof of \cite[Lemma 3(i),(iii)]{AMT2} and using the fact 
that the 
distance between a point and a nonempty, closed and convex subset of a Hilbert space is attained),
it can be shown that if $p_k,\, k\in K=\{1,2\}$ are two different
points in a Hilbert space $X$ and if $(R_1,R_2)$ is a zone diagram
in $X$, then the boundaries $C_k$ of $R_k$ represent the components
of a trisector with respect to $P_k=\{p_k\}$, i.e., $C_k=\{x\in X:
d(x,P_k)=d(x,C_j\}$ for $k\neq j\in K$. This conclusion extends the
existence part of \cite[Theorem 1]{AMT2}. Actually, by different
arguments this fact can be generalized to other spaces and to
more general sets $P_k$. 
\end{rem}
\section{Uniqueness}\label{sec:uniqueness}
As the examples given in Sections \ref{sec:examples} and
\ref{sec:exmetric} show, a zone diagram is not unique in general. In
spite of this, it is possible to formulate several
necessary and sufficient conditions for uniqueness. We first state and prove a general uniqueness theorem for antimonotone mappings.
\begin{thm}\label{thm:AntiMonotoneUnique}
Let $(Y,\leq)$ be a partially ordered set, and let $T:Y\to Y$ be antimonotone. If $\,T^2$ has fixed points
$m$ and $M$ with the property that $m\leq \mu\leq M$ for any other fixed point $\mu$ of $\,T^2$,
then the following conditions are equivalent and each of them
suffices for $T$ to have exactly one fixed point.
\begin{enumerate}[(a)]
\item\label{item:mM} $m=M$.
\item\label{item:T2} $T^2$ has a unique fixed point.
\item\label{item:T2inT1} Any fixed point of $T^2$ is a fixed point of $T$.
\item\label{item:fixT2_is_fixT1} The fixed point sets of $T$ and $T^2$ coincide.
\item\label{item:mM_fix_T}  Either $m$ or $M$ is a fixed point of $T$.
\end{enumerate}
If, in addition, $(Y,\leq)$ is a complete lattice, then all these conditions are equivalent to the following one:
\begin{equation}\tag{\emph{f}}\label{item:AB}
 A\leq B\,\, \textrm{for any} \,\,A,B\in Y\,\, \textrm{which satisfy}\,\,\, T^2(B)\leq B\,\, \textrm{and}\,\, A\leq T^2(A).
\end{equation}
\end{thm}
\begin{proof}
\textbf{\eqref{item:mM}$\Rightarrow$\eqref{item:T2}:} If $\mu=T^2\mu$, then by assumption $m\leq \mu\leq M=m$, i.e., $\mu=m=M$ is the unique fixed point of $T^2$.\vspace{0.2cm}\\
\eqref{item:T2}$\Rightarrow$\eqref{item:T2inT1}: If $\mu=T^2\mu$, then $T\mu=T^3\mu=T^2(T\mu)$, i.e., $T\mu$ is a fixed point of $T^2$, so by uniqueness, $T\mu=\mu$.\vspace{0.2cm}\\
\eqref{item:T2inT1}$\Rightarrow$\eqref{item:fixT2_is_fixT1}: One inclusion holds by assumption, and the other holds in general, since if $\mu=T\mu$, then $\mu=T\mu=T^2\mu$, so $\mu$ is also a fixed point of $T^2$.\vspace{0.2cm}\\
\eqref{item:fixT2_is_fixT1}$\Rightarrow$\eqref{item:mM_fix_T}: 
obvious.\vspace{0.2cm}\\
\eqref{item:mM_fix_T}$\Rightarrow$\eqref{item:mM}: 
Suppose, for example, that $TM = M$. Since $m \le M$, the antimonotonicity
of $T$ implies that $Tm \geq TM = M$. Since $Tm$ is a fixed point of 
$T^2$, it follows that $Tm = M$. Hence $m = T^2m = TM = M$, that is, 
$m = M$. \vspace{0.2cm}\\
\eqref{item:fixT2_is_fixT1},\eqref{item:T2}$\Rightarrow$ $T$ has 
a unique fixed point: obvious.\vspace{0.2cm}

Finally, we will show that \eqref{item:mM} and \eqref{item:AB} are equivalent:\vspace{0.2cm}\\
\eqref{item:AB}$\Rightarrow$\eqref{item:mM} Since $M\leq T^2M$ and $T^2m\leq m$, it follows that $M\leq m$, but $m\leq M$ and hence there is equality.\vspace{0.1cm}\\
\eqref{item:mM}$\Rightarrow$\eqref{item:AB} 
Suppose $A,B\in Y$ satisfy $T^2(B)\leq B$ and $A\leq T^2(A)$.
Since the pair $(Y,\leq)$ is a complete lattice, the Knaster-Tarski fixed
point theorem implies that  $m=\bigwedge\{y\in Y: T^2(y)\leq y\}$ and $M=\bigvee\{y\in Y: y\leq T^2(y)\}$. Hence $A\leq M=m\leq B$.

\end{proof}
Since by Theorem \ref{thm:DoubleZoneDiagram} the $\Dom$ mapping satisfies the conditions of Theorem \ref{thm:AntiMonotoneUnique}, we obtain several equivalent sufficient conditions for the uniqueness of zone diagrams. This again shows the importance of the concept of a double zone diagram. In fact, because of the special structure of the Dom mapping we can get a stronger result which will be formulated as a special corollary below. Before formulating this corollary, we  note that, in fact, the implication \eqref{item:T2}$\Rightarrow$ ``$T$ has a unique fixed point'' is true without any assumption on $T$ and $X$ and without the partial order.

\begin{cor}\label{cor:unique_n2}
Let $(X,d)$ be an $m$-space, $P = (P_k)_{k \in K}$ a tuple of nonempty 
sets in $X$, and let $T=\Dom$. Then each one of the six conditions in
Theorem \ref{thm:AntiMonotoneUnique} suffices for $X$ to have
exactly one zone diagram with respect to $P$. If, in addition, $K = 
\{1,2\}$, then these conditions are also necessary.
\end{cor}
\begin{proof}
In view of Theorem \ref{thm:AntiMonotoneUnique} and the above discussion, 
only the last assertion remains to be proven.
So suppose that $K = \{1,2\}$ and that $T$ has a unique fixed
point. We will show that part \eqref{item:T2inT1} of 
Theorem \ref{thm:AntiMonotoneUnique} holds. 
To this end, let $Z=(Z_1,Z_2)$ be any fixed point of $T^2$. Then
\begin{equation*}
Z_1=\dom(P_1,\dom(P_2,Z_1))=g(Z_1),\quad
Z_2=\dom(P_2,\dom(P_1,Z_2))=h(Z_2),
\end{equation*}
 where  $T_k(A)=\dom(P_k,A)$, $k = 1, 2$, $g=T_1\circ T_2$ and $h=T_2\circ 
T_1$. Let
$Y_1=\{A\in \mathcal{P}^{*}(X): P_1\subseteq A\}$.
By Lemma \ref{lem:DomProperties}, $g$ is a monotone mapping which 
maps $Y_1$ into itself and $Y_1$ is a complete lattice. Therefore Theorem 
\ref{thm:CSB} implies that
$g$ has a least and a greatest fixed points $m_1$ and $M_1$, respectively. 
By defining $m_2:=T_2m_1$ and $M_2:=T_2M_1$, we obtain
that $(m_1,m_2)$ and $(M_1,M_2)$ are fixed points of $T$ (since, for instance,  $m_1=T_1(T_2m_1)=T_1m_2$), so by
uniqueness $m_1=M_1$. 
 This implies that $g$ has a
unique fixed point in $Y_1$, so $Z_1=m_1=M_1$. In the same way, $h$ has a unique
 fixed point, and since both $Z_2$ and $m_2$ are fixed points of $h$, they must coincide, i.e., $Z=(Z_1,Z_2)=(m_1,m_2)$ is a fixed point of $T$.
\end{proof}
To illustrate an application of this corollary, let $X=[-3,3]$ with 
$d(x,y)=|x-y|$, and let $P=(P_1,P_2)=(\{-3\},\{3\})$.
 For $t\in \N$ define $R^{t}=\Dom^{t}(P)$. A short calculation shows that
 $\Dom^2 (X,X)= \Dom P=([-3,0],[0,3])$.  More generally, 
we obtain $\Dom^{t}P=([-3, 
a_{t}],[b_{t},3])$, where  $a_0=-3,b_0=3$ and $a_{t+1}=(b_{t}-3)/2,\,b_{t+1}=(a_{t}+3)/2$. By elementary considerations,  $\{a_{2t}\}$ increases to -1, $\{b_{2t}\}$ decreases to 1, $\{a_{2t+1}\}$ decreases to  -1 and $\{b_{2t+1}\}$ increases to 1. Hence  $\bigcup_{t=0}^{\infty}\Dom^{2t}P$ increases to $([-3,-1),(1,3])$, and  $\bigcap_{t=0}^{\infty}\Dom^{2t+1}P$ decreases to $([-3,-1],[1,3])$. Since by Lemma \ref{lem:DomProperties}\eqref{item:Y}, a double zone diagram $R=(R_1,R_2)$ satisfies $P\subseteq R\subseteq X$, by repeated iterations we obtain $\bigcup_{t=0}^{\infty}\Dom^{2t}P\subseteq R \subseteq
 \bigcap_{t=0}^{\infty}\Dom^{2t}(X,X)=\bigcap_{t=0}^{\infty}\Dom^{2t+1}P$, and using the fact that $R_1,R_2$ are obviously closed we
  have that $R=([-3,-1],[1,3])$. This indeed proves the uniqueness of the 
double zone diagram, and hence, by Corollary \ref{cor:unique_n2}, of the 
zone diagram. A similar 
consideration shows the uniqueness of the zone diagram in 
Example \ref{ex:rings}.

\section{The finite case}\label{sec:algorithm}
In this short section we describe a practical way of constructing a double zone diagram of arbitrary number of sites 
and a zone diagram of 2 sites in the particular case where $X$ is a finite set.
In both cases the construction is by iteration, and we give some estimates on the number
of required iterations.
\begin{thm}\label{thm:FiniteZone}
Let $(X,d)$ be a finite $m$-space and let $(P_1,P_2)\in
\mathcal{P}^{*}(X)\times \mathcal{P}^{*}(X)$. 
Let $g_1=T_1\circ T_2$ and $g_2=T_2\circ T_1$, where
$T_k(A)=\dom(P_k,A)$. Then there are nonnegative integers $n_1,N_1
\leq |X|-|P_1|$ and $n_2,N_2\leq |X|-|P_2|$  such that 
\begin{equation*}
\begin{array}{ll}
   R=({g_1}^{n_1}(P_1),T_2(g^{n_1}(P_1)),  & S=({g_1}^{N_1}(X),T_2(g_1^{N_1}(X)), \\
   Z=(T_1({g_2}^{n_2}(P_2)),{g_2}^{n_2}(P_2)), & W=(T_1({g_2}^{N_2}(X)),{g_2}^{N_2}(X))
\end{array}
\end{equation*}
are all zone diagrams in $X$.
\end{thm}
\begin{proof}
We will show this for the first case. The other cases are proved
similarly. Since $g_1$ is monotone and $P_1\subseteq g_1(P_1)$, it
follows that the sequence $a_{t}=g_1^{t}(P_1)
,t\in \N\bigcup\{0\}$ is increasing. Since $\mathcal{P}^{*}(X)$ is finite,
the sequence becomes constant starting from some index $n_1$. One
can give the following linear estimate instead of the ``default'' exponential one by
 observing that the sequence $b_{t}=|X|-|g_1^{t}(P_1)|$ is a
decreasing sequence of nonnegative integers, and hence also becomes
constant starting from $n_1\leq b_0$. At this point
$\{a_{t}\}_{t}$ becomes constant. Let $R_1:=a_{n_1}$.
Then $R_1=a_{n_1}=a_{n_1+1}=g_1(R_1)$, and hence $R_1$ is a fixed
point of $g_1$, i.e., $R_1=\dom(P_1,\dom(P_2,R_1))$. Letting
$R_2:=T_2(R_1)=\dom(P_2,R_1)$, we obtain that $(R_1,R_2)$ is a zone
diagram in $X$.
\end{proof}
\begin{thm}\label{thm:FiniteDoubleZone}
Let $(X,d)$ be a finite $m$-space and let $P=(P_k)_{k\in K}$, $K$ finite, be a given tuple
of nonempty sets in $X$. Then there are nonnegative integers $n_1,N_1\leq \sum_{k\in K}(|X|-|P_k|)$ such that
\begin{equation*}
R=\Dom^{2n_1}P,\quad S=\Dom^{2N_1}(X)_{k\in K}
\end{equation*}
are double zone diagrams in $X$.
\end{thm}
\begin{proof}
Take $Y$ from Lemma \ref{lem:DomProperties}\eqref{item:Y}. 
Then $\Dom^2 P, \Dom^2(X)_{k\in K}\in Y$, so  $P\subseteq \Dom^2 P$ and 
$\Dom^2(X)_{k\in K}\subseteq (X)_{k\in K}$. 
Since the mapping $\Dom^2$ is monotone, 
the sequence $\{\Dom^{2t}P\}_{t=1}^{\infty}$ is increasing 
and the sequence $\{\Dom^{2t}(X)_{k\in K}\}_{t=1}^{\infty}$ is decreasing. 
Since $Y$   is finite, these sequences become constant from a certain 
point on, and this constant must be a fixed point of $\Dom^2$. As for the 
estimate, by defining the sequence $a_t=\sum_{k\in K}|(\Dom^{2t}P)_k|$, we see
that $\{a_t\}_{t=0}^{\infty}$ is increasing and can take integer values between $\sum_{k\in K}|P_k|$ and $\sum_{k\in K}|X|$, and when it becomes constant, so does $\{\Dom^{2t}P\}_{t=1}^{\infty}$.
\end{proof}
The above algorithm  may be applied to the (approximate) construction of 
zone diagrams
in normed spaces, by considering a large finite set of points
(grid) there, and constructing the zone diagram as above. However,
it should be checked in what sense the resulting zone
diagram is close (perhaps with respect to the Hausdorff distance) to the 
real one(s). This algorithm can also be used for finding a stable 
configuration of the combinatorial game described in Section 
\ref{sec:combi}.\\

\section{Concluding remarks and open problems}
The most interesting open problem is whether existence holds for
a general cardinal $n=|K|$. Our method of proof for the case $n=2$ cannot
be carried over to the general case. This can be clearly illustrated already in the case
$n=3$. By Theorem \ref{thm:DoubleZoneDiagram}, we do know that a double 
zone diagram $R=(R_1,R_2,R_3)$ exists, and by definition it satisfies 
the following system of equations:
\begin{equation*}
\begin{array}{l}
R_1=\dom(P_1,\dom(P_2,R_3\,\textstyle{\bigcup}\, R_1)\,\textstyle{\bigcup}\,\dom(P_3,R_1\,\textstyle{\bigcup}\, R_2)),\\
R_2=\dom(P_2,\dom(P_3,R_1\bigcup R_2)\bigcup\dom(P_1,R_2\bigcup R_3)),\\
R_3=\dom(P_3,\dom(P_1,R_2\bigcup R_3)\bigcup\dom(P_2,R_3\bigcup R_1)).
\end{array}
\end{equation*}
 Unfortunately, in contrast to the case $n=2$, the above system is coupled, and
 it is not clear how to obtain a zone diagram $Z=(Z_1,Z_2,Z_3)$ from $R$. We will, however, describe several possible ways to construct a zone diagram and  point out the difficulties
  we have encountered.

The first way is to use the double zone diagram $R$ in a similar way to the one used in the proof of Theorem \ref{thm:n_is_2}, by defining $Z_1:=R_1,\,Z_2:=R_2$ and $Z_3:=\dom(P_3,Z_1\bigcup Z_2)$. However, it is not clear why $(Z_1,Z_2,Z_3)$ is a zone diagram because  it is not
clear from the above system why $Z_1=\dom(P_1,Z_2\bigcup Z_3)$ and $Z_2=\dom(P_2,Z_3\bigcup Z_1)$.

In the second way we again use the double zone diagram $R$. 
We define $Z_1:=R_1$, plug $R_1$ in the second and third equations, and
then define $Z_2,Z_3$ as solutions of the resulting system:
\begin{equation*}
\begin{array}{l}
Z_2=\dom(P_2,\dom(P_3,R_1\bigcup Z_2)\bigcup\dom(P_1,Z_2\bigcup Z_3)),\\
Z_3=\dom(P_3,\dom(P_1,Z_2\bigcup Z_3)\bigcup\dom(P_2,Z_3\bigcup R_1)).
\end{array}
\end{equation*}
This system does have  a solution because the mapping on the right hand side is monotone.
However, it is not clear why $(Z_1,Z_2,Z_3)$ will be a zone diagram. A similar phenomenon occurs
if one takes $Z_1:=R_1,Z_2:=R_2$ and defines $Z_3$ as (one of) the solution(s) of the resulting equation.

The third way  is to look at the system defining the zone diagram:
\begin{equation*}
R_1=\dom(P_1,R_2\,\textstyle{\bigcup}\, R_3),\,R_2=\dom(P_2,R_3\,\textstyle{\bigcup}\,
R_1),\,R_3=\dom(P_3,R_1\,\textstyle{\bigcup}\, R_2).
\end{equation*}
(Now $R$ is no longer the above double zone diagram.)
 We  eliminate one of the unknowns, say $R_1$, and arrive at the
 system
\begin{equation*}
R_2=\dom(P_2,R_3\, \textstyle{\bigcup}\, \dom(P_1,R_2\,\textstyle{\bigcup}\, R_3)),\, R_3=\dom(P_3,R_2\,\textstyle{\bigcup}\, \dom(P_1,R_2\,\textstyle{\bigcup}\, R_3)).
\end{equation*}
If we could show that the above system has a solution $(R_2,R_3)$, then by defining
$R_1:=\dom(P_1,R_2\bigcup R_3)$ we would indeed obtain that $R=(R_1,R_2,R_3)$ is a zone diagram in $X$.
 Unfortunately, the mapping on the right hand side is no longer
monotone, so it is not clear why there exists a solution.

Therefore, in order to prove existence for general $n$ one has to
adopt other strategies, or to somehow modify the above ones. At the
moment we have several partial results in specific cases (a class of
normed spaces) which are in preparation, but the
general case of $m$-spaces is open, and even in the case where $X$ is
a finite metric space the situation is not clear (but we feel that the
combinatorial interpretation may help here). Anyway, we conjecture
that existence holds in any $m$-space, at least for finite
$n$. The infinite case may be problematic as the remark in
Example \ref{ex:ProbInfinite} shows, but we conjecture that
existence holds in any metric space for any cardinal if the sets
$P_k$ are ``nice'' and ``far enough'' from each other.

The question of uniqueness is also interesting. We conjecture that
uniqueness holds at least in finite dimensional normed spaces, assuming the sets $P_k$ are ``nice'' and ``far enough'' from each other. It is also
 of interest to determine whether the second assertion in Corollary \ref{cor:unique_n2} holds for general $n$.  Theorem  \ref{thm:AntiMonotoneUnique} shows that uniqueness arguments may be a strategy for proving existence.
 A related question is to what extent uniqueness is the probable case,
at least for finite metric spaces embedded in normed spaces. For
instance, Example \ref{ex:simple} shows that if $X$ is embedded in
the interval [-1,1], then uniqueness holds with the exception of one
case ($a=0$), so it holds with probability 1. The cases where it
does not hold may be analogous in some sense to eigenvalues of a
linear operator, and it would be of interest to investigate them.

Finally, it would be interesting to find some applications of zone
diagrams to other parts of mathematics and elsewhere. We think that
zone diagrams do have this potential, say in optimization theory and
computer science, and even in the natural sciences (see \cite[pp. 
1183-4]{AMTn} or \cite[p. 341]{AMT2}), including the general case of
$m$-spaces.  A reader who is familiar with Voronoi diagrams might
have noticed that given a tuple $P=(P_k)_{k\in K}$, $\Dom(P)$ is
nothing but the Voronoi diagram induced by $P$. Hence Voronoi
diagrams are related to zone diagrams, and since they have many
applications, including in the case of generalized distances (see
\cite{BregmanVoronoi} and the references therein), this may also be
true for zone diagrams. The half-spaces $Eq^+(x,y)$ and $Eq^-(x,y)$ in the Hilbert ball
with the hyperbolic metric, which appear in \cite[pp. 112-115]{GoebelReich}, provide another example of a dominance region, and thus may be related to zone diagrams. We also note
that the combinatorial interpretation and the examples given in the
beginning of the paper may also point to some applications.\\\\


{\noindent }\textbf{Acknowledgments}\vspace{0.2cm}\\ The second author 
was partially supported by
the Fund for the Promotion of Research at the Technion and by the
Technion President's Research Fund.
Both authors thank the referee for several helpful comments and 
suggestions. 



\end{document}